\documentclass{amsart}
\usepackage{graphicx} 
\usepackage[english]{babel}
\usepackage[utf8]{inputenc}

\usepackage{tabularx}

\usepackage{graphicx}
\usepackage{setspace}
\usepackage{subfigure}
\usepackage{wrapfig}
\usepackage[dvipsnames]{xcolor} 
\usepackage{listings}

\usepackage{amsmath}
\usepackage{amsfonts}
\usepackage{amssymb}
\usepackage{amsthm}
\usepackage{dlfltxbcodetips}
\usepackage{mathtools}
\usepackage{stix} 
\usepackage{tikz}

\definecolor{LozTop}{named}{red}
\definecolor{LozLeft}{named}{violet}
\definecolor{LozRight}{named}{green}

\newcommand{\loz}[3][0.95]{%
  \tikz[baseline=-0.6ex, scale=#1, x=0.866ex, y=1ex]{
    \begin{scope}[rotate=#2]
      \path[draw=black, line width=0.25pt, fill=#3]
        (0,1) -- (1,0.5) -- (0,0) -- (-1,0.5) -- cycle;
    \end{scope}
  }%
}

\newcommand{\lozH}[1][1.5]{\loz[#1]{0}{LozTop}}     
\newcommand{\lozL}[1][1.5]{\loz[#1]{60}{LozLeft}}   
\newcommand{\lozR}[1][1.5]{\loz[#1]{-60}{LozRight}} 
\newcommand{\lozHlow}[1][1.5]{\raisebox{-0.7ex}{\lozH[#1]}}

\numberwithin{equation}{section}

\newcommand*{\N}{\mathbb{N}}
\newcommand*{\Z}{\mathbb{Z}}
\newcommand*{\Q}{\mathbb{Q}}

\usepackage{thmtools}

\declaretheorem[
	name=Theorem,
	numberwithin=section
	]{thm}
\declaretheorem[
	name=Lemma,
	sibling=thm,
	]{lem}
\declaretheorem[
	name=Proposition,
	sibling=thm,
	]{prop}
\declaretheorem[
	name=Corollary,
	sibling=thm,
	]{cor}

\declaretheorem[
	name=Definition,
	style=definition,
	numbered=no,
	]{defin}

\declaretheorem[
	name=Notation,
	style=definition,
	numbered=no,
	]{nota}

\declaretheorem[
	name=Remark,
	style=remark,
	numbered=no
	]{rem}	
\declaretheorem[
	name=Example,
	style=remark,
	numbered=no
	]{exam}

\title{The 1/3-phenomenon of placement probabilities of tilings in the semiregular hexagon}

\author{Marcus Schönfelder}
\thanks{{marcus.schoenfelder@univie.ac.at}. The author was supported by the Austrian Science Foundation FWF, grant 10.55776/F1002, in the framework of the Special Research Programm ``Discrete Random Structures: Enumeration and Scaling Limits".}

\date{May 2026}

\lstdefinelanguage{MathematicaCustom}{
    morekeywords={Plot,Sin,Cos,Tan,Log,Exp,Table,Manipulate, Simplify, Zb, Rate, Sum, Binomial, Pochhammer, Factorial},
    sensitive=true,
    morestring=[b]"
}

\begin{document}

\begin{abstract}
    We prove Krattenthaler's conjecture from 2001 about the $1/3$-phenomenon for lozenge tilings of semiregular hexagons. In a first step we reduce the problem to the case of regular hexagons. In a second step we further reduce the question to a special case already covered in the literature. This is achieved by vast application of the celebrated Zeilberger Algorithm and the theory of holonomic functions.  
\noindent  \end{abstract}
\maketitle

\section{Introduction}
\subsection{Background and motivation}
The dimer model has become one of the central objects of study in statistical mechanics. It concerns the behavior of random dimer configurations on graphs. In the language of combinatorics, a \textit{dimer configuration} of a graph $G=(V,E)$ is a \textit{perfect matching} of its vertices, i.e., a collection of edges $\mu\subseteq E$ such that every vertex $v\in V$ is incident to exactly one edge $e\in \mu$. If the graphs are subgraphs of the square or hexagonal lattices, perfect matchings are in bijection with tilings of regions that are dual to the graph. The most prominent examples are domino and rhombus (lozenge) tilings. For rhombus tilings, one seeks to cover a region formed by unit triangles with rhombus-shaped tiles, also called \textit{lozenges}. An example is shown in Figure~\ref{domi1}. The connection with perfect matchings is illustrated in Figure~\ref{domi2}.

\begin{figure}
\centering
\includegraphics[width=0.4\textwidth]{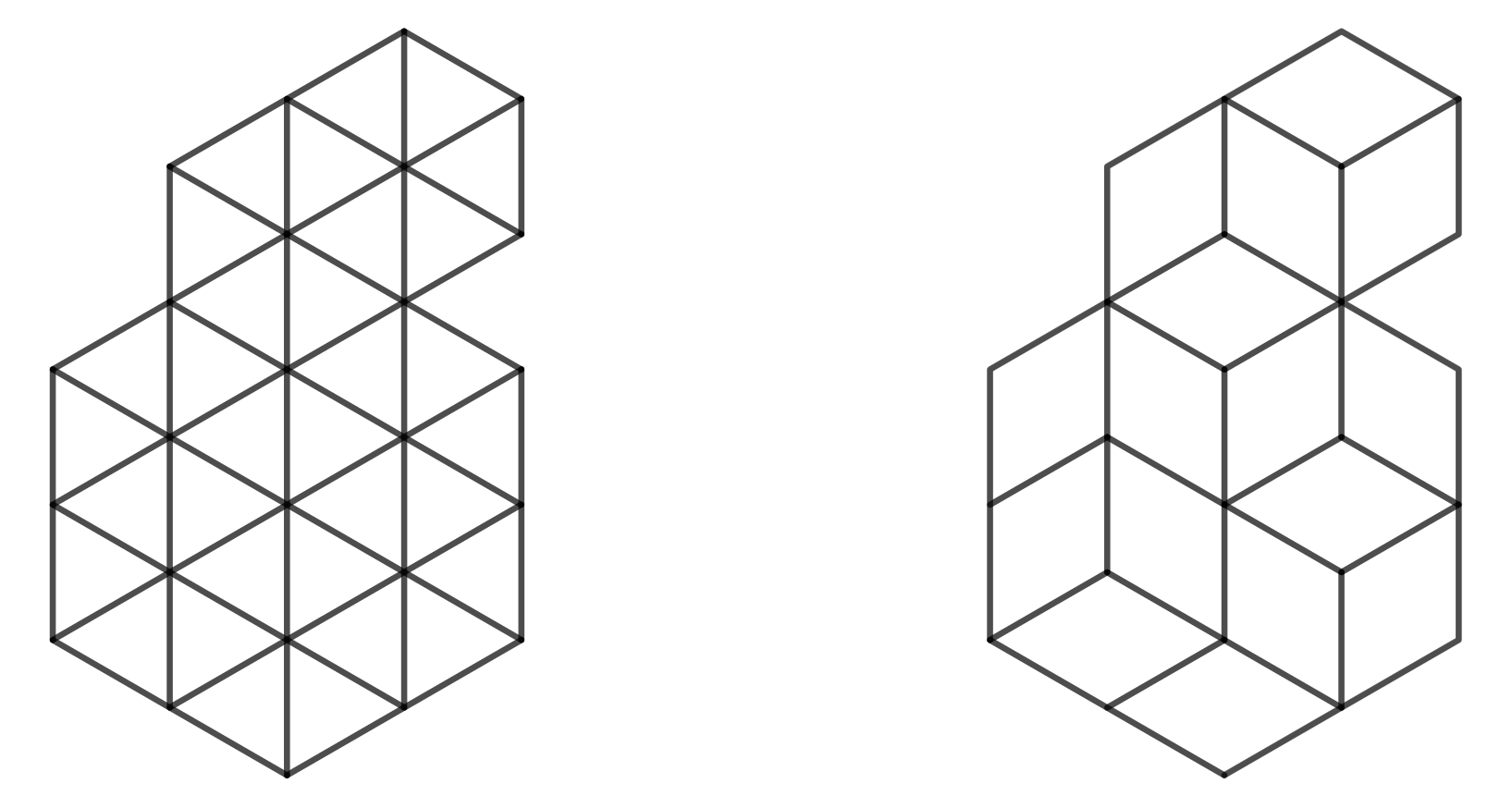}
\caption{A lozenge region left and one of its possible rhombus tilings on the right.}
\label{domi1}
\end{figure}

\begin{figure}
\centering
\includegraphics[width=0.2\textwidth]{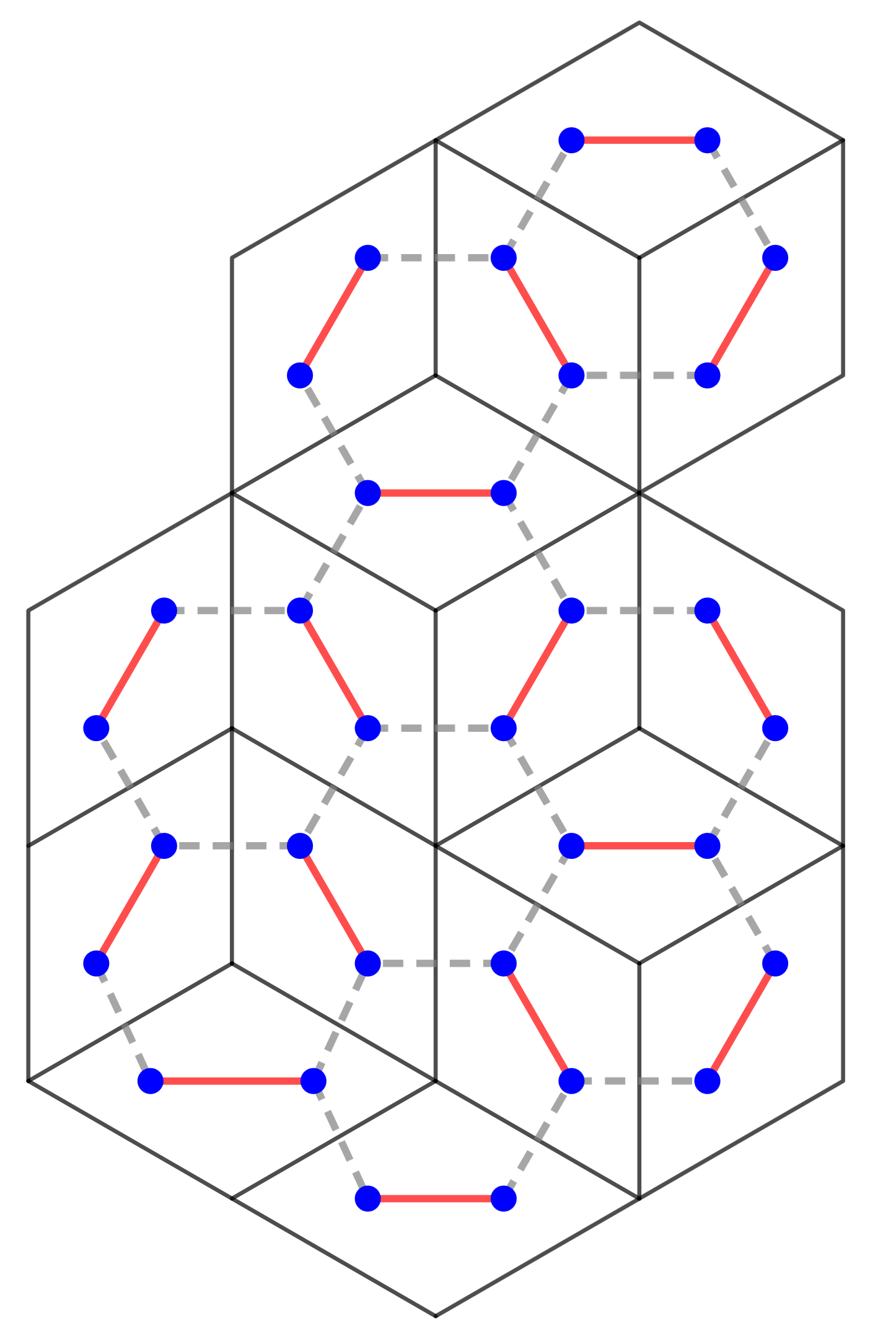}
\caption{The bijection between rhombus tilings and perfect matchings of the dual graph.}
\label{domi2}
\end{figure}

Advances in this theory have gone hand in hand with the development of new enumerative techniques. For rhombus tilings, MacMahon’s Box Formula~\cite{MacMahon} counts tilings of semiregular hexagons and relates them to plane partitions. The semiregular hexagon $H_{a,b,c}$ with side lengths $a,b,c$ is a hexagonal region of the triangular lattice, as shown in Figure~\ref{hexa1}. MacMahon’s Box Formula then states that the number of tilings of $H_{a,b,c}$ is given by
 \[\prod_{i=1}^a\prod_{j=1}^b\prod_{k=1}^c\frac{i+j+k-1}{i+j+k-2}.\]

When sampling a lozenge tiling of $H_{a,b,c}$ uniformly at random one can ask about the \textit{placement probability} at a certain position, namely the probability of observing a given tile at a specified location in such a random tiling. Equivalently, this number yields for a specified edge the probability that it is contained in a random perfect matching.

For the semiregular hexagon, closed formulas for the placement probabilities were developed by Fischer~\cite{FischerHexaLoch}, Johansson~\cite{johannson,GUE}, Gilmore~\cite{gilmoreInv}, and, in greater generality, Petrov~\cite{Petrov}. However, the complexity of these formulas makes them cumbersome for many applications. In 2001, Krattenthaler~\cite{1/3} conjectured a substantial simplification of their structure: he argued that, in a semiregular hexagon whose size tends to infinity the placement probability of a single lozenge should be approximately $1/3$ plus an error term. Moreover, he predicted that this error term is a fairly explicit expression whose form depends on the fixed location of the lozenge under consideration. An analogous phenomenon was proved in~\cite{1/4} for domino tilings of the Aztec diamond.

In this article, we prove Krattenthaler’s conjectured $1/3$‑phenomenon. The precise setup is explained below, and the main result is stated in Theorem~\ref{1/3}.

\subsection{Setup and main result}
Consider lozenge tilings of the semiregular hexagon $H_{a,b,c}$ with integer side lengths $a,b,c,a,b,c$, as shown in Figure~\ref{hexa1}. Such a region allows three types of lozenges: horizontal \lozHlow, right-leaning \lozL, and left-leaning \lozR. One easily checks that the positions of the horizontal lozenges already uniquely determine the entire tiling.
\begin{figure}
\centering
\includegraphics[width=0.4\textwidth]{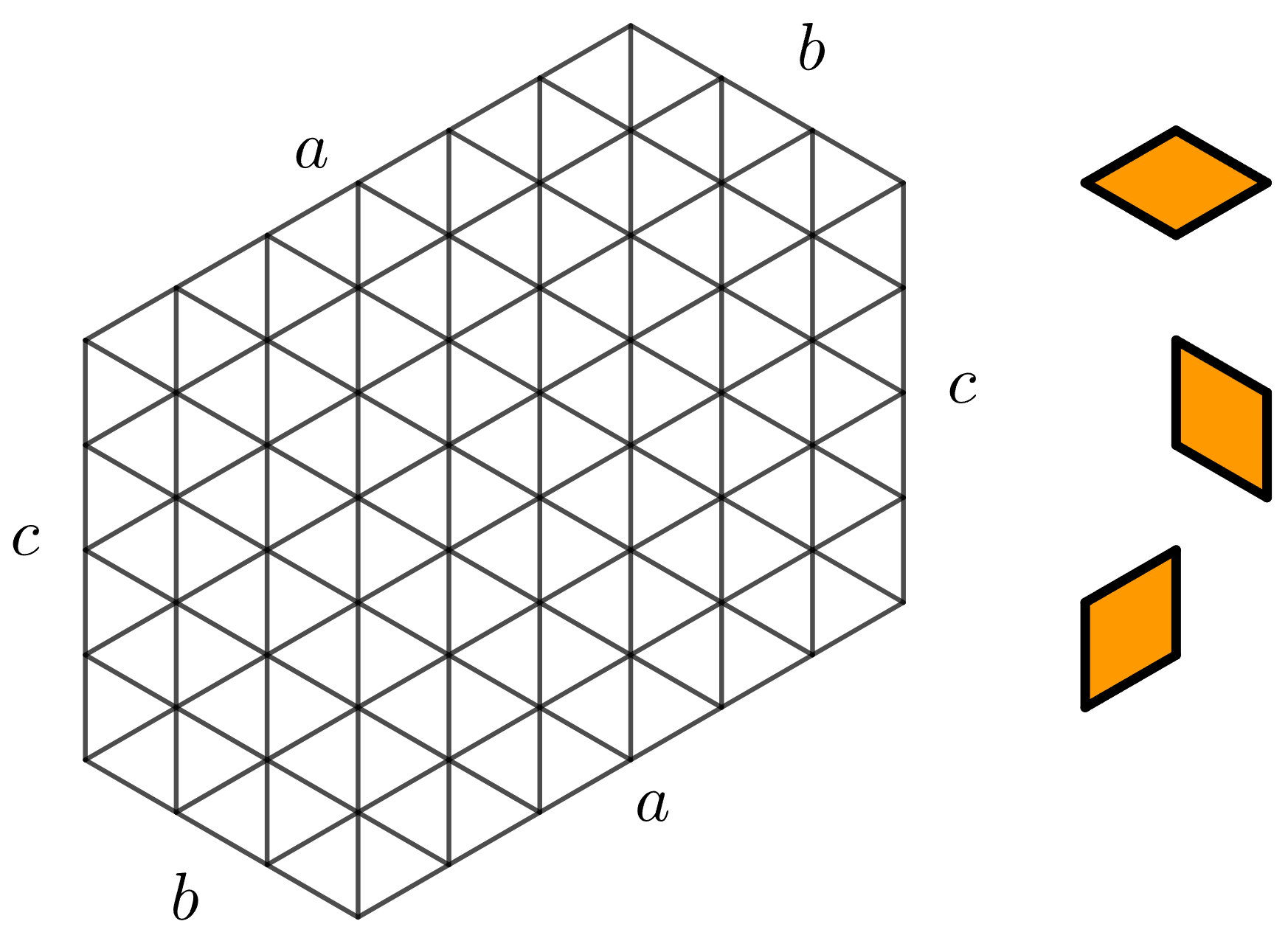}
\caption{The region $H_{a,b,c}$ has to be tiled with three types of lozenges. In the picture we have $a=6$, $b=3$ and $c=4$.}
\label{hexa1}
\end{figure}

We are interested in placement probabilities, in particular the probability of observing a horizontal lozenge at a given position in a uniformly random lozenge tiling of $H_{a,b,c}$. To study this statistic, we parametrize the possible positions of a horizontal lozenge by a suitable coordinate system.

\begin{nota}
    We denote by $\mathbb{P}(a,b,c,x,y;n)$ the probability to observe a horizontal lozenge in a random tiling of the region $H_{a+2n,b+2n,c+2n}$ at position $(x+2n,y+2n)$ according to the coordinate system described in Figure~\ref{coord}. 
\end{nota}
\begin{figure}
\centering
\includegraphics[width=0.4\textwidth]{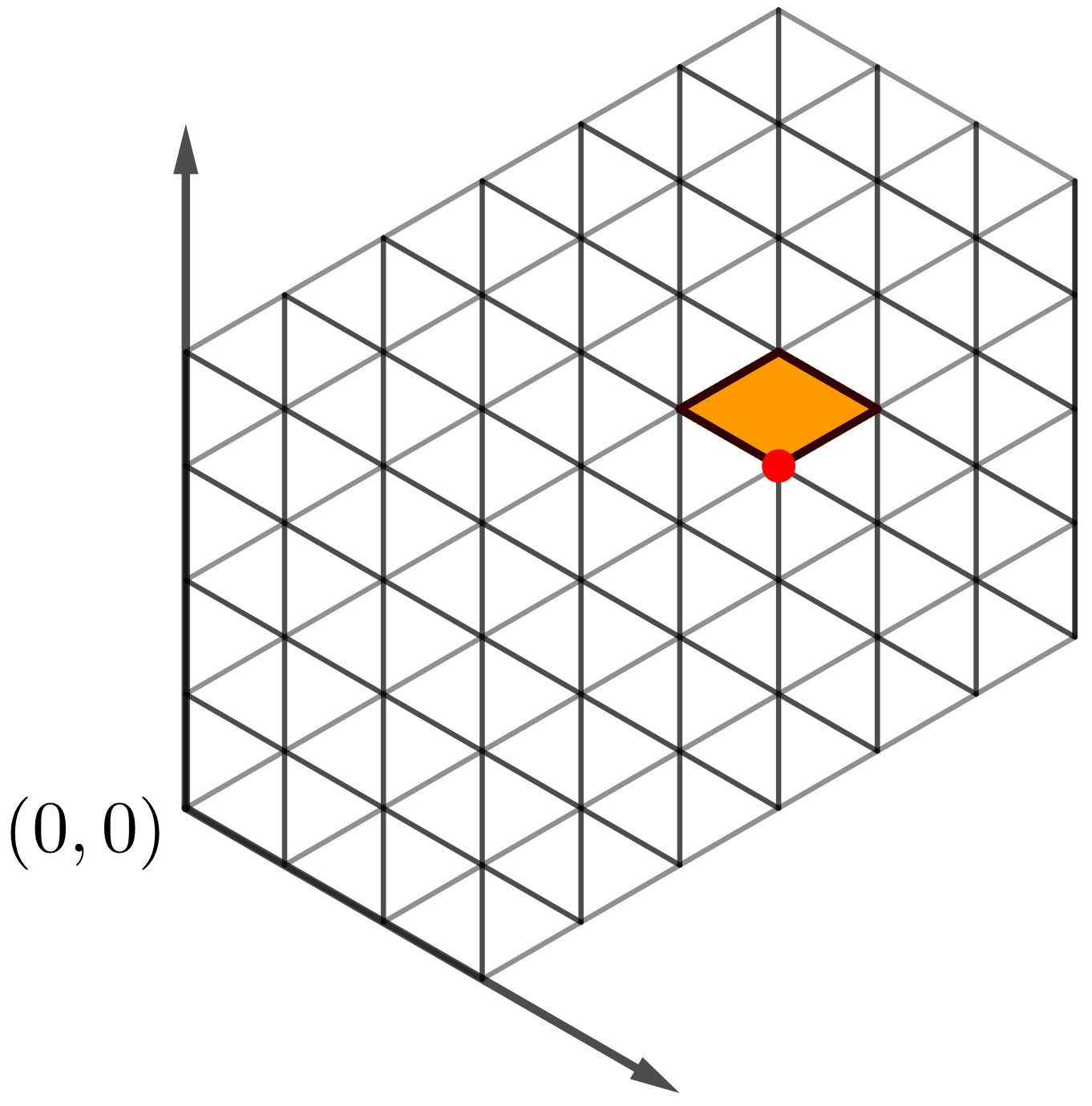}
\caption{Parametrising the position of horizontal lozenges: We set the south-west corner of the hexagon $H_{a,b,c}$ to be the origin. Moreover let the two adjacent sides generate the axis of our coordinate system. The position $(x,y)\in\Z$ of the lozenge marked in orange is then given by the location of its south-most point in this coordinate system. In this particular case, we have marked the lozenge at position $(6,6)$.}
\label{coord}
\end{figure}

Now, Krattenthaler's conjectural $1/3$-phenomenon, first stated in \cite{1/3}, proposes the following.

\begin{thm}[\sc{The $1/3$-phenomenon}]\label{1/3}
    For all $a,b,c,x,y\in\Z$ and all positive integers $n$ with $n\geq N_{a,b,c,x,y}\in\N$ large enough we have
    \[\mathbb{P}(a,b,c,x,y;n)=\frac{1}{3}+f_{a,b,c,x,y}(n)\cdot\binom{2n}{n}^3\bigg/\binom{6n+2}{3n+1}\]
    where $f_{a,b,c,x,y}(n)$ is a rational function in $n$.
\end{thm}

\begin{exam}
    For each instance of $f_{a,b,c,x,y}(n)$, Theorem \ref{1/3} together with MacMahon's Box Formula yields a compact enumeration formula counting the number of lozenge tilings in $H_{a+2n,b+2n,c+2n}$ with a horizontal tile fixed at $(x+2n,y+2n)$. Below, we list several cases in which $f_{a,b,c,x,y}(n)$ is known. More examples can also be found in the original article of Krattenthaler~\cite{1/3}.
    \begin{align}
        f_{-1,-1,0,-1,-1}(n)=f_{-1,0,0,0,-1}(n)=f_{1,2,1,1,0}(n)={}&0 \hphantom{aaaaa} \text{for }n\geq 1,\notag\\
        f_{1,1,0,0,-1}(n)=f_{1,1,0,0,0}(n)=f_{2,0,1,1,0}(n)={}&\frac{1}{3}\hphantom{aaaaa} \text{for }n\geq 1,\notag\\
        f_{0,0,0,0,0}(n)={}&-\frac{(6n+1)}{6(3n+1)}\hphantom{aaaaa} \text{for }n\geq 1,\notag\\
        f_{2,1,0,3,-1}(n)={}&\frac{4n^3+18n^2+12n+1}{6(n + 1)^2
(2n - 1)}\hphantom{aaaaa} \text{for }n\geq 2.\notag
\end{align}
\end{exam}

In 2025—twenty-four years after Krattenthaler formulated his conjecture—we proved a related phenomenon for domino tilings of the Aztec diamond (see \cite{1/4}). There we derived a formula analogous to Theorem~\ref{1/3} by establishing linear recurrences for the placement probabilities with respect to the individual parameters. In the Aztec diamond, the Domino Shuffling algorithm for generating tilings uniformly at random (see \cite{Azdiam2,jockusch} for details) was the central tool that reduced the problem to special instances. For the semiregular hexagon, there also exist Markov-chain algorithms that sample tilings uniformly at random; see, in particular, the work of Gorin and Borodin~\cite{hexshuffle}. However, these sampling algorithms are much more involved than Domino Shuffling and thus were not suitable for our purpose. Still, we will prove the $1/3$-phenomenon by establishing linear recurrences. As opposed to \cite{1/4}, here these will be obtained via Zeilberger’s algorithm~\cite{AB} and, more generally, the theory of holonomic functions, with substantial support from computer algebra systems.

First, in Section~2 we present a brief summery of the basic theory of Zeilberger’s algorithm and the idea of Creative Telescoping. We also explain how the Mathematica packages \texttt{fastZeil.m} by Paule and Schorn~\cite{Paule} and \texttt{HolonomicFunctions.m} by Koutschan~\cite{holo,KoutschanThesis,Koutschan} are used to support our proof.  

In Section~3, we turn to new results. Using a triple-sum formula for placement probabilities due to Fischer~\cite{FischerHexaLoch}, together with Zeilberger’s algorithm, we reduce the proof of the $1/3$-phenomenon to the case of regular hexagons, i.e., to the case in the theorem where $a=b=c=0$. 

Finally, in Section~4, we apply Koutschan's algorithms to further reduce the problem to the case $x=y=0$. Since $\mathbb{P}(0,0,0,0,0;n)$ has already been studied in the literature and has been proved to exhibit the $1/3$-phenomenon by Krattenthaler and Fulmek~\cite{CenterCase}, this completes the proof of Theorem~\ref{1/3}.

\section{The Zeilberger Algorithm and the theory of holonomic functions}\label{ZAHA}
In this section, we provide basic information about the algorithmic tools we use to prove the 1/3-phenomenon. Readers familiar with the Mathematica packages \texttt{fastZeil.m}~\cite{Paule} and \texttt{HolonomicFunctions.m}~\cite{holo,KoutschanThesis,Koutschan} may skip this section. For others, we briefly review the core ideas of Zeilberger's algorithm for hypergeometric functions and give some context motivating further interest. Moreover, we explain how to interpret the outputs of these algorithms, in case the reader wishes to reproduce the computer-assisted parts of our proof.

A sequence $(F(n))_{n\in\N}$ is called \textit{holonomic} or \textit{P-recursive} if its elements satisfy a linear recurrence with polynomial coefficients. That is, there exist  polynomials $a_0(n),\dots,a_d(n)$ such that
\[a_0(n)F(n)+a_1(n)F(n-1)+a_2F(n-2)+\dots +a_d(n)F(n-d)=0\]
for all $n\geq d$, where $d\in \N$ is the order of the recurrence.
The Zeilberger algorithm, or more generally the theory of holonimic functions, is designed to find recurrence relations for holonomic sequences automatically using computer algebra systems. In particular, all hypergeometric series are holonomic with respect to each of their parameters. 

In this section we present the general ideas of the theory of holonomic functions in the hypergeometric case. The exposition is based primarily on the book "A=B" by Petkovšek, Wilf and Zeilberger \cite{AB}. We start with a basic definition.
 \begin{defin}[\cite{AB}]
     A function $F(n,k)$ is called \textit{proper hypergeometric} if 
     \[F(n,k)=p(n,k)\frac{\prod_{i=1}^I(a_in+b_ik+c_i)!}{\prod_{j=1}^J(u_jn+v_jk+w_j)!}x^k,\]
     where $x$ is an indeterminate over the complex numbers, $p(n,k)$ is a bivariate polynomial, all the $a_i,b_i,c_i,u_j,v_j$ and $w_j$ are specific integers and $I$ and $J$ are positive integers. We call 
     \[H(n)=\sum_{k=-\infty}^\infty F(n,k)\] a hypergeometric series.
 \end{defin}
 \begin{rem}
     Notice in particular that $F(n+1,k)/F(n,k)$ and $F(n,k+1)/F(n,k)$ are both rational functions.
 \end{rem}
The idea of Zeilberger’s algorithm is to find linear recurrences for the summands $F(n,k)$, which are often easier to determine, and then lift them to recurrences for $H(n)$. In \cite[Thm. 6.2.1]{AB} it is shown that for a proper hyper geometric term $F(n,k)$ there always exist polynomials $a_0(n),a_1(n),\dots, a_d(n)$ and a function $G(n,k)$, with $G(n,k)/F(n,k)$ rational in $n$ and $k$, such that
\[a_0(n)F(n,k)+a_1(n)F(n+1,k)+\dots +a_d(n)F(n+d,k)=G(n,k+1)-G(n,k).\]
It is further explained how to compute both $G(n,k)$ and the coefficients $a_j(n)$. 
Now, if $G(n,k)$ has finite support, summing over all integers $k$ will then lead to
\[a_0(n)H(n)+a_1(n)H(n+1)+\dots +a_d(n)H(n+d)=0\]
since the right-hand side telescopes. Even if $H(n)=\sum_{k=0}^N F(n,k)$ is only a finite sum, we at least obtain an inhomogeneous linear recurrence by summing over all values $k=0,\dots,N$:
\[a_0(n)H(n)+a_1(n)H(n+1)+\dots +a_d(n)H(n+d)=G(n,N+1)-G(n,0).\]
As a by-product, implementations of Zeilberger's algorithm also return a rational function $R(n,k)=G(n,k)/F(n,k)$, called the \textit{proof certificate}. Therefore, to verify the resulting recurrence for $H(n)$, it suffices to check that
\[a_0(n)F(n,k)+\dots +a_d(n)F(n+d,k)=R(n,k+1)F(n,k+1)-R(n,k)F(n,k)\]
which, after dividing both sides by $F(n,k)$,  reduces to comparing two rational functions.

We will see in the following sections that the placement probabilities $\mathbb{P}(a,b,c,x,y;n)$ can be expressed as hypergeometric sums and products. Hence, it is natural to apply the theory of holonomic functions to these quantities. We use the Mathematica implementation \texttt{fastZeil.m} by Paule and Schorn~\cite{Paule}, which can be obtained from the website of the Research Institute for Symbolic Computation (RISC) at the Johannes Kepler University Linz. Its main command is \texttt{Zb[…]}, which applies Zeilberger’s algorithm. 
Applying \texttt{Zb[...]} expects the syntax
\[\texttt{Zb[}F(n,k),\{k,\textit{min, max}\},n,\deg\texttt{]}\]
where \textit{min} and \textit{max} specify the lower and upper bounds of the summation index $k$, and $\deg$ is the maximum order of the recurrence to be found. If one wants to see the certificate, one has to enter \texttt{show[R]}. 

However impressive, this algorithm helps only to some extent. Once the expressions become too complicated, Zeilberger’s algorithm may fail to produce an output before our computational resources are exhausted. When this happens, we switch to the algorithms provided by \texttt{HolonomicFunctions.m} by Koutschan~\cite{KoutschanThesis,Koutschan}. This package can be obtained from the same website as \texttt{fastZeil.m}. The main difference is that we lift the problem to the more flexible language of operators.
\begin{nota}
    Let $F(x)$ be a function in the integer variable $x$. We denote by $S_{x}$ the \textit{shift operator of} $x$. That is, we have
    \[S_{x}F(x)=F(x+1).\]
\end{nota}
In this notation the recurrence relation
\[a_0(n)F(n,k)+a_1(n)F(n+1,k)+\dots +a_d(n)F(n+d,k)=G(n,k+1)-G(n,k)\]
becomes
\[\big(a_0(n)+a_1(n)S_n+\dots +a_d(n)S_n^d\big)F(n,k)=(S_k-1)G(n,k).\]
Observe that the operator expression on the left is now a polynomial in $S_n$. Therefore, instead of seeking recurrences, we may equivalently search for operators in the annihilator ideal within the Ore algebra generated by the shift operators. This is precisely what the command \texttt{Annihilator[...]} does: it returns is a Gröbner basis for the annihilator ideal. It works via calling
\[\texttt{Annihilator[}\textit{Expression}(x_1,x_2,\dots,x_l),\{\texttt{S[}x_1\texttt{]}\dots,\texttt{S[}x_l\texttt{]}\}\texttt{]}\]
where $\textit{Expression}(x_1,x_2,\dots,x_l)$ is some holonomic expression in the variables $x_1$ up to $x_l$ and the set in the second argument contains the shift operators that span the ideal in which we are looking for annihilators.

One further algorithm we will use frequently in Section~\ref{centreproof} is the \texttt{CreativeTelescoping} command, which is useful for eliminating parameters. We use it in the form
\[\texttt{CreativeTelescoping[}GB, S-1,\{\textit{remaining OPs}\}\texttt{]}\]
where $GB$ is a Gröbner basis of an annihilating ideal, $S$ is the telescoping shift operator and $\{\textit{remaining OPs}\}$ is the set of operators that remain. The output consists of two lists of operators $\{A_1,\dots,A_l\}$ and $\{B_1,\dots,B_l\}$ such that for each $i=1,\dots, l$ we have
\[A_i+(S-1)B_i \in GB\]
and the $A_i$'s involve only the operators in $\{\textit{remaining OPs}\}$. In other words, the command\break \texttt{CreativeTelescoping} produces annihilating operators of the specified telescoping form. In the following sections we will see very useful and rather involved applications of these algorithms.

\section{Reduction to the regular case}

This section is dedicated to reducing the problem of the 1/3-phenomenon to the case of regular hexagons. In particular, we prove the following theorem.

\begin{thm}[\sc{Reduction to the regular case}]
    If the $1/3$-phenomenon holds for all positions $(x+2n,y+2n)$ inside the regular hexagon $H_{2n,2n,2n}$, then it is true for arbitrary positions inside the semiregular hexagon $H_{a+2n,b+2n,c+2n}$ \/for all $a,b,c\in\Z$. In other words, if $\mathbb{P}(0,0,0,x,y;n)$ is of the shape described in Theorem~\ref{1/3} for all $x,y\in\Z$ then so is $\mathbb{P}(a,b,c,x,y;n)$. 
\end{thm}

First, since Theorem~\ref{1/3} requires the specific form of $\mathbb{P}(a,b,c,x,y;n)$ only for sufficiently large $n$, we may assume that all parameters $a,b,c,x,y$ are non-negative. Namely, suppose $c = \min\{a,b,c,x,y\} < 0$ and choose $n_0 \in \mathbb{N}$ such that $2n_0 \geq |c|$. Then, for all $n \ge n_0$, we have
\begin{align}
    \mathbb{P}(a,b,c,x,y;n)&=\mathbb{P}(\text{horizontal lozenge at }(x+2n,y+2n)\text{ in }H_{a+2n,b+2n,c+2n})\notag\\
    &=\mathbb{P}(\text{horizontal lozenge at }(x'+2n',y'+2n')\text{ in }H_{a'+2n',b'+2n',c'+2n'})\notag\\
    &=\mathbb{P}(a',b',c',x',y';n')\notag
\end{align}
where $a'=a+2n_0,\ b'=b+2n_0,\ c'=c+2n_0,\ x'=x+2n_0,\  y'=y+2n_0$ and $n'=n-n_0$ with $a',b',c',x',y'\geq 0$. Thus, if Theorem~\ref{1/3} holds for all positive choices of arguments, it also holds in general. Therefore, from now on we assume $a,b,c,x$ and $y$ to be non-negative.

As already mentioned in the introduction, there are several formulas known to express $\mathbb{P}(a,b,c,x,y;n)$. For our purposes, the formula due to Fischer is the most convenient to work with.
\begin{thm}[{\cite[Lem. 2]{FischerHexaLoch}}]\label{fischer}
    Let $a,b,c$ be positive integers and let $x,y$ be integers such that $0\leq x\leq b+a-1$ and $0\leq y\leq c+a-1$. Then the probability that a uniformly chosen random tiling of $H_{a,b,c}$ contains the horizontal lozenge tile with bottom-most vertex at $(x,y)$ is equal to
    \begin{align}
        \frac{c!}{(b+1)_c}\sum_{i=1}^{a}\sum_{j=1}^{a}\sum_{s=1}^{j}(-1)^{i+s}\binom{j-1}{s-1}&\binom{c+i+x-y-2}{x-1}\binom{b+s-x+y-1}{b+s-x-1}\notag\\
        &\cdot \frac{(b+1)_{s-1}(c+1)_{i-1}(b+c+i)_{j-i}}{(j-i)!(i-1)!(b+c+1)_{s-1}}.\notag
    \end{align}
    Here, $m_k=m(m+1)\dots(m+k-1)$ denotes the Pochhammer symbol.
\end{thm}

Therefore, by inserting the correct arguments in Fischer's formula, we obtain the following expression for the quantity in question.
\begin{cor} \label{corform}
    We have
    \begin{align}
        \mathbb{P}&(a,b,c,x,y;n)=\notag \\
        &\frac{(c+2n)!}{(b+2n+1)_{c+2n}}\sum_{i=1}^{a+2n}\sum_{j=1}^{a+2n}\sum_{s=1}^{j}(-1)^{i+s}\binom{j-1}{s-1}\binom{c+2n+i+x-y-2}{x+2n-1}\notag\\
        &\hphantom{aaa} \cdot\binom{b+2n+s-x+y-1}{b+s-x-1} \frac{(b+2n+1)_{s-1}(c+2n+1)_{i-1}(b+c+4n+i)_{j-i}}{(j-i)!(i-1)!(b+c+4n+1)_{s-1}}.\notag
    \end{align}
\end{cor}

Our goal is to use this expression for $\mathbb{P}(a,b,c,x,y;n)$ to develop recursive arguments that allow us to reduce the parameters $a,b$ and $c$ to zero, one by one, while showing that all this reduction steps preserve the 1/3-phenomenon. As a first central insight, we notice that the formula presented in Theorem~\ref{fischer} is actually a sum of products. Namely, we have
\begin{align}
     &\frac{c!}{(b+1)_c}\sum_{i=1}^{a}\sum_{j=1}^{a}\sum_{s=1}^{j}(-1)^{i+s}\binom{j-1}{s-1}\binom{c+i+x-y-2}{x-1}\binom{b+s-x+y-1}{b+s-x-1}\notag\\
        &\hphantom{aaaaaaaaaaaaaaaaaaaaa}\cdot \frac{(b+1)_{s-1}(c+1)_{i-1}(b+c+i)_{j-i}}{(j-i)!(i-1)!(b+c+1)_{s-1}}\notag\\
         &=\frac{c!}{(b+1)_c}\sum_{j=1}^{a}\Bigg(\sum_{i=1}^{a}(-1)^{i}\binom{c+i+x-y-2}{x-1}\frac{(c+1)_{i-1}(b+c+i)_{j-i}}{(j-i)!(i-1)!}   \Bigg)\notag\\
        &\hphantom{aaaaaaaaaaaa}\cdot \Bigg(\sum_{s=1}^{j} (-1)^s\binom{j-1}{s-1}\binom{b+s-x+y-1}{b+s-x-1}\frac{(b+1)_{s-1}}{(b+c+1)_{s-1}}     \Bigg).\notag
\end{align}

Hence, if we set
\[I_{a,b,c,x,y}(j;n):=\sum_{i=1}^{a+2n}(-1)^{i}\binom{c+2n+i+x-y-2}{x+2n-1}\frac{(c+2n+1)_{i-1}(b+c+4n+i)_{j-i}}{(j-i)!(i-1)!}\] 
and
\[F_{a,b,c,x,y}(j;n):=\sum_{s=1}^{j} (-1)^s\binom{j-1}{s-1}\binom{b+2n+s-x+y-1}{b+s-x-1}\frac{(b+2n+1)_{s-1}}{(b+c+4n+1)_{s-1}}\]
by Corollary~\ref{corform} we obtain
\begin{align}
    \mathbb{P}(a,b,c,,x,y;n)=\frac{(c+2n)!}{(b+2n+1)_{c+2n}}\sum_{j=1}^{a+2n} I_{a,b,c,x,y}(j;n)\cdot F_{a,b,c,x,y}(j;n). \label{IFform}
\end{align}

Knowing this, we are ultimately able to start the reduction of the parameters.

\subsection{Reduction of the parameter $a$}
The reduction of the first parameter turns out to be the most involved part. The essential idea is to develop linear recurrences with respect to each of the parameters and then prove that these recurrences preserve the shape of the formula we are looking for. Sometimes, we are able to exhibit these recurrences by hand, however at some point a repeated application of Zeilberger's algorithm as described in \cite{AB} via a computer algebra system becomes inevitable. 

An initial key result in this section is the following lemma.

\begin{lem}\label{redA}
    For all $a>0$, all $b,c,x,y$ non-negative integers and $n\in\N$ large enough it holds that
    \begin{align}
        \mathbb{P}(a,b,c,x,y;n)&-\mathbb{P}(a-1,b,c,x,y;n)\notag\\
        &=\frac{(c+2n)!}{(b+2n+1)_{c+2n}}I_{a,b,c,x,y}(a+2n;n)\cdot F_{a,b,c,x,y}(a+2n;n).\notag
    \end{align}
\end{lem}

\begin{proof}
    First, notice that $F_{a,b,c,x,y}(j;n)$ is in fact not dependent on $a$. Thus, we might as well write $F_{b,c,x,y}(j;n)$. Nevertheless, the situation is different for the function $I_{a,b,c,x,y}(j;n)$. Here we observe
    \begin{align}
        I_{a,b,c,x,y}&(j;n)=\sum_{i=1}^{a+2n}(-1)^{i}\binom{c+2n+i+x-y-2}{x+2n-1}\frac{(c+2n+1)_{i-1}(b+c+4n+i)_{j-i}}{(j-i)!(i-1)!}\notag \\
        &=\sum_{i=1}^{a-1+2n}(-1)^{i}\binom{c+2n+i+x-y-2}{x+2n-1}\frac{(c+2n+1)_{i-1}(b+c+4n+i)_{j-i}}{(j-i)!(i-1)!}\notag\\
        &+(-1)^{a+2n}\binom{a+c+4n+x-y-2}{x+2n-1}\frac{(c+2n+1)_{a+2n-1}(a+b+c+6n)_{j-a-2n}}{(j-a-2n)!(a+2n-1)!}.\notag
    \end{align}
    Note that the sum over $i$ in the last expression equals $I_{a-1,b,c,x,y}(j;n)$. However the last summand is special. By properties of factorials of negative integers, it vanishes whenever $j<a+2n$ since then
    \[(j-a-2n)!=\infty.\]
    Therefore, we have $I_{a,b,c,x,y}(j;n)=I_{a-1,b,c,x,y}(j;n)$ unless $j=a+2n$. Using this and Formula~\eqref{IFform} we have
    \begin{align}
        \mathbb{P}(a,b,c,x,y;n)&-\mathbb{P}(a-1,b,c,x,y;n)=\notag\\
        &\frac{(c+2n)!}{(b+2n+1)_{c+2n}}\Bigg(I_{a,b,c,x,y}(a+2n;n)\cdot F_{b,c,x,y}(a+2n,n) \hphantom{a}\notag \\
        &\hphantom{aaaaaaa}+\sum_{j=1}^{a-1+2n}\Big(\underbrace{I_{a,b,c,x,y}(j;n)-I_{a-1,b,c,x,y}(j;n)}_{=0}\Big)F_{b,c,x,y}(j;n)\Bigg)\notag\\
        &=\frac{(c+2n)!}{(b+2n+1)_{c+2n}}I_{a,b,c,x,y}(a+2n;n)\cdot F_{b,c,x,y}(a+2n,n) \notag
    \end{align}
    as desired.
\end{proof}

Lemma~\ref{redA} is the key to the reduction of the problem to the case where $a=0$. Namely, if we suppose Theorem~\ref{1/3} holds and therefore $\mathbb{P}(a,b,c,x,y;n)$ equals $1/3$ plus some expression as predicted in Theorem~\ref{1/3}, as well as $\mathbb{P}(a-1,b,c,x,y;n)$ equals $1/3$ plus such an expression, then their difference will be of similar kind. On the other hand, if we prove that the difference between those two placement probabilities is of the shape predicted by Theorem~\ref{1/3}, meaning we show
\[\frac{(c+2n)!}{(b+2n+1)_{c+2n}}I_{a,b,c,x,y}(a+2n;n)\cdot F_{a,b,c,x,y}(a+2n;n)=g_{a,b,c,x,y}(n)\binom{2n}{n}^3\bigg/\binom{6n+2}{3n+1}\]
for some rational function $g_{a,b,c,x,y}(n)$ in $n$, then we can inductively show that $\mathbb{P}(a,b,c,x,y;n)$ fulfils Theorem~\ref{1/3} if $\mathbb{P}(0,b,c,x,y;n)$ satisfies it. Hence, our next big step is to establish the proposition below.

\begin{prop}\label{IFshape}
    For all fixed $a,b,c,x,y\in \N\cup\{0\}$ and $n\in \N\cup\{0\}$ large enough we have the following.
    \begin{itemize}
        \item[(i)] There exists a rational function $q_{a,b,c,x,y}(n)$ in $n$ such that
        \[I_{a,b,c,x,y}(a+2n;n)=(-1)^n\frac{(3n+1)!}{(n!)^3}\cdot q_{a,b,c,x,y}(n).\]
        \item[(ii)] There exists a rational function $r_{a,b,c,x,y}(n)$ in $n$ such that
        \[F_{a,b,c,x,y}(a+2n;n)=(-1)^n\frac{(2n)!(3n+1)!(4n)!}{(n!)^3 (6n+2)!}\cdot r_{a,b,c,x,y}(n).\]
    \end{itemize}
\end{prop}

In the proof of the proposition we will apply the following easy lemma at several occasions.

\begin{lem} \label{recpreserv}
    Let $\big(F_k(n)\big)_{k\in\N\cup\{0\}}$ be a family of functions which satisfies a linear recurrence of degree $d\in\N$ with rational coefficients, i.e. we have
    \[0=p_{k,0}(n)F_k(n)+p_{k,1}(n)F_{k-1}(n)+\dots +p_{k,d}(n)F_{k-d}(n)\]
    where $p_{k,i}(n)$ are rational functions with $p_{k,0}(n)\neq 0$. If there is an (arbitrary) function $C(n)$ in $n$ such that there exists a rational function $q_i(n)$ with
    \[F_i(n)=C(n)\cdot q_i(n)\]
    for all $i=0,1,2,\dots, d-1$, then there exist rational functions $q_j(n)$ such that \\$F_j(n)=C(n)\cdot q_j(n)$ for all $j\in \N\cup\{0\}$.
\end{lem}

\begin{proof}
    By induction on $k$ we have
    \begin{align}
        F_k(n)&=\frac{-1}{p_{k,0}(n)}\Big(p_{k,1}(n)F_{k-1}(n)+\dots +p_{k,d}(n)F_{k-d}(n)\Big)\notag\\
        &=\frac{-1}{p_{k,0}(n)}\Big(p_{k,1}(n)\cdot C(n)\cdot q_{k-1}(n)+\dots +p_{k,d}(n)\cdot C(n)\cdot q_{k-d}(n)\Big)\notag\\
        &=C(n)\cdot\bigg(-\frac{p_{k,1}(n)\cdot q_{k-1}(n)+\dots +p_{k,d}(n)\cdot q_{k-d}(n)}{p_{k,0}(n)}\bigg),\notag
    \end{align}
    where the expression in the parenthesis defines the rational function $q_k(n)$.
\end{proof}

Next comes the proof of Proposition~\ref{IFshape} which contains the promised application of the Zeilberger Algorithm.

\begin{proof}[Proof of Proposition \ref{IFshape}]
    Given a function $H(l)=\sum_{k=0}^\infty Z(l,k)$ where all the summands $Z(l,k)$ are proper hypergeometric terms Zeilberger's algorithm delivers a linear recurrence with polynomial coefficients for $H(l)$ in $l$. Implemented in a computer it does this automatically using symbolic computations. 

    In our case, it is easy to check, that both $I_{a,b,c,x,y}(a+2n;n)$ and $F_{a,b,c,x,y}(a+2n;n)$ are hypergeometric in all their arguments. Hence, we apply Zeilberger's algorithm to
    \begin{align}
        I_{a,b,c,x,y}&(a+2n;n)=\notag\\ 
        &\sum_{i=1}^{a+2n}(-1)^{i}\binom{c+2n+i+x-y-2}{x+2n-1}\frac{(c+2n+1)_{i-1}(b+c+4n+i)_{a+2n-i}}{(a+2n-i)!(i-1)!}\notag\\
        =&\sum_{i=1}^{\infty}(-1)^{i}\binom{c+2n+i+x-y-2}{x+2n-1}\frac{(c+2n+1)_{i-1}(b+c+4n+i)_{a+2n-i}}{(a+2n-i)!(i-1)!}\notag
    \end{align}
    and
    \begin{align}
        F_{a,b,c,x,y}&(a+2n;n)=\notag \\ &\sum_{s=1}^{a+2n} (-1)^s\binom{a+2n-1}{s-1}\binom{b+2n+s-x+y-1}{b+s-x-1}\frac{(b+2n+1)_{s-1}}{(b+c+4n+1)_{s-1}}\notag\\
        =&\sum_{s=1}^{\infty} (-1)^s\binom{a+2n-1}{s-1}\binom{b+2n+s-x+y-1}{b+s-x-1}\frac{(b+2n+1)_{s-1}}{(b+c+4n+1)_{s-1}}\notag
    \end{align}
    with respect to all parameters $a,b,c,x$ and $y$. 
    Doing so, we obtain the following recurrences for $I_{a,b,c,x,y}(a+2n;n)$.
    \begin{itemize}
        \item With respect to the parameter $a$ we get
        \begin{align}
            &(a+b+c+6n)(a+b+2n-x)\cdot I_{a,b,c,x,y}(a+2n;n)\hphantom{a}   \label{arec}\\
            &+(-1-3a-2a^2-2b-2ab-c-2ac-bc-10n - 12 a n - 4 b n - 4 c n - 12 n^2 + x  \hphantom{a}\notag\\
            &+a x + c x +4 n x + a y + b y + 4 n y)\cdot I_{a+1,b,c,x,y}(a+1+2n;n)  \hphantom{a}\notag\\
            &+(1 + a + 2 n) (2 + a + c + 2 n - y)\cdot I_{a+2,b,c,x,y}(a+2+2n;n) =0\notag
        \end{align}
        with certificate 
        \[R_{I,a}(a,b,c,x,y,n,i)=\frac{(-1 + i) (-1 + b + c + i + 4 n) (a + b + c + 6 n) (-c - i + y)}{(1 + a - i + 2 n) (2 + a - i + 2 n)}.\]
        \item With respect to the parameter $b$ we have
        \begin{align}
            &(a + b + c + 6 n) (a + b + 2 n - x)\cdot I_{a,b,c,x,y}(a+2n;n)\hphantom{a} \label{brec}\\
            &+(-1 - a - 2 b - 2 a b - 2 b^2 - c - b c - 6 n - 4 a n - 12 b n - 
  12 n^2 + x + a x + 2 b x + c x  \hphantom{a}\notag\\
            &+8 n x - a y - b y - 4 n y)\cdot I_{a,b+1,c,x,y}(a+2n;n)\hphantom{a}\notag\\
            &+(1 + b + 2 n) (1 + b + 2 n - x + y)\cdot I_{a,b+2,c,x,y}(a+2n;n)=0\notag
        \end{align}
        with certificate
        \[R_{I,b}(a,b,c,x,y,n,i)=\frac{(-1 + i) (a + b + c + 6 n) (-c - i + y)}{b + c + i + 4 n}.\]
        \item With respect for the parameter $c$ we obtain
        \begin{align}
            &(a + b + c + 6 n) (1 + c + 2 n + x - y)\cdot I_{a,b,c,x,y}(a+2n;n)\hphantom{a}\\
            &+(-3 - 3 a - 2 b - 5 c - 2 a c - b c - 2 c^2 - 18 n - 4 a n - 12 c n - 
  12 n^2 - x - a x - c x \notag\\
            &- 4 n x +2 y + a y + b y + 2 c y + 8 n y)\cdot I_{a,b,c+1,x,y}(a+2n;n)\hphantom{a}\notag\\
            &+(2 + c + 2 n) (2 + a + c + 2 n - y)\cdot I_{a,b,c+2,x,y}(a+2n;n)=0\notag
        \end{align}
        with certificate 
        \begin{align}
            R_{I,c}(a,&b,c,x,y,n,i)=\frac{(-1+i)}{(1 + c + 2 n) (2 + c + 2 n) (b + c + i + 4 n) (-1 - c - i + y)}\notag\\
            &\cdot\big(i^2 (a + 2 n) (2 + c + 2 n) (a + b + c + 6 n) + 
            i (2 + c + 2 n) (a + b + c + 6 n) \notag\\
            &\cdot(-1 - 2 b - c + 2 a c - b c - 
            4 n + 4 a n + 4 c n + 12 n^2 + x + a x \notag\\
            &+ c x + 4 n x - a y + 
            b y) \notag\\
            &+ (2 + c + 2 n) (a + b + c + 6 n) (-b - c - 3 b c - c^2 + 
            a c^2 - b c^2 - 4 n - 4 b n \notag\\
            &- 8 c n + 4 a c n + 2 c^2 n - 8 n^2 + 
            4 a n^2 + 4 b n^2 + 16 c n^2 + 24 n^3 - b x + c x + a c x \notag\\
            &+ c^2 x + 2 a n x + 2 b n x + 6 c n x + 12 n^2 x + 2 b y - a c y + 
            2 b c y + 4 n y - 2 a n y + 2 b n y \notag\\
            &+ 2 c n y + b x y + 2 n x y - 
            b y^2 - 2 n y^2)\big)\notag.
        \end{align}
        \item With respect to $x$ we have
        \begin{align}
            &(1 + c + 2 n + x - y) (-1 + b + 2 n - x + y)\cdot I_{a,b,c,x,y}(a+2n;n)\hphantom{a}\\
            &+(3 - a - 2 b + c - b c + 2 n - 4 b n - 4 n^2 + 5 x - a x - 2 b x + 
  c x + 2 x^2 - 2 y  \hphantom{a}\notag\\
            &+a y + b y - 2 x y)\cdot I_{a,b,c,x+1,y}(a+2n;n)\hphantom{a}\notag\\
            &+(-2 + a + b + 2 n - x) (1 + 2 n + x)\cdot I_{a,b,c,x+2,y}(a+2n;n)=0\notag
        \end{align}
        with certificate
        \[R_{I,x}(a,b,c,x,y,n,i)=\frac{(-1 + i) (-1 + b + c + i + 4 n) (-c - i + y)}{2 n + x}.\]
        \item Finally for the parameter $y$ we obtain
        \begin{align}
            &(a + c + 2 n - y)  (2 n + y)\cdot I_{a,b,c,x,y}(a+2n;n)\hphantom{a}\\
            &+(1 - a - c - b c - 2 n - 4 c n - 4 n^2 - x + a x + c x + 2 y - a y  \hphantom{a}\notag\\
            &+b y - 2 c y - 2 x y + 2 y^2)\cdot I_{a,b,c,x,y+1}(a+2n;n)\hphantom{a}\notag\\
            &+(-1 + c + 2 n + x - y) (1 + b + 2 n - x + y)\cdot I_{a,b,c,x,y+2}(a+2n;n)=0\notag
        \end{align}
        with certificate
        \[R_{I,y}(a,b,c,x,y,n,i)=\frac{(-1 + i) (-1 + b + c + i + 4 n) (-1 + 2 n + x) (-c - i + y)}{(1 - c - i - 2 n - x + y) (2 - c - i - 2 n - x + y)}.\]
    \end{itemize}
    For the function $F_{a,b,c,x,y}(a+2n;n)$ we compute another set of five different recurrences using Zeilberger's algorithm.
    \begin{itemize}
        \item For the parameter $a$ we have
        \begin{align}
            &(a + 2 n) (-1 + a + c + 2 n - y)\cdot F_{a,b,c,x,y}(a+2n;n)\hphantom{a}\\
            &+(-a - 2 a^2 - 2 a b - c - 2 a c - b c - 2 n - 12 a n - 4 b n - 
                 4 c n - 12 n^2 + a x + c x + 4 n x  \hphantom{a}\notag\\
            &+y + a y + b y + 4 n y)\cdot F_{a+1,b,c,x,y}(a+1+2n;n)\hphantom{a}\notag\\
            &+(1 + a + b + c + 6 n) (1 + a + b + 2 n - x)\cdot F_{a+2,b,c,x,y}(a+2+2n;n)=0\notag
        \end{align}
        with certificate
        \[R_{F,a}(a,b,c,x,y,n,s)=\frac{(-1 + s) (a + 2 n) (-1 + b + c + s + 4 n) (1 - b - s + x)}{(1 + a - k + 2 n) (2 + a - k + 2 n)}.\]
        \item For the parameter $b$ we observe
        \begin{align}
            &(1 + b + c + 4 n) (2 + b + c + 4 n) (1 + b + 2 n - x + y)\cdot F_{a,b,c,x,y}(a+2n;n)\hphantom{a}\\
            &-(2 + b + c + 4 n)  (2 + 3 a + 4 b + 2 a b + 2 b^2 + c + b c + 14 n + 
             4 a n + 12 b n + 12 n^2\hphantom{a}\notag\\
             &- 2 x - a x - 2 b x - c x - 8 n x + y + a y + b y + 4  n  y)\cdot F_{a,b+1,c,x,y}(a+2n;n)\hphantom{a}\notag\\
            &+(2 + b + 2 n) (1 + a + b + c + 6 n) (1 + a + b + 2 n - x)\cdot F_{a,b+2,c,x,y}(a+2n;n)=0\notag
        \end{align}
        with certificate
        \begin{align}
        R_{F,b}&(a,b,c,x,y,n,s)=\frac{(-1+s)}{(1 + b + 2 n) (2 + b + 2 n) (b + c + k + 4 n) (-b - k + x)}\notag\\
        &\cdot \big(s^2 (a + 2 n) (2 + b + 2 n) (1 + b + c + 4 n) (2 + b + c + 4 n) + 
 s (2 + b + 2 n)\notag\\
 &\cdot(1 + b + c + 4 n) (2 + b + c + 4 n) (2 a b - c - 
    b c + 4 a n + 4 b n + 12 n^2 - a x\notag\\
    &+ c x + y + a y + b y + 
    4 n y) + (2 + b + 2 n) (1 + b + c + 4 n) (2 + b + c + 
    4 n)\notag\\& (a b^2 - b c - b^2 c + 4 a b n + 2 b^2 n + 4 n^2 + 4 a n^2 + 
    16 b n^2 + 4 c n^2 + 24 n^3 - a b x \notag\\
    &+ c x + 2 b c x + 2 n x - 
    2 a n x + 2 b n x + 2 c n x - c x^2 - 2 n x^2 + b y + a b y + 
    b^2 y \notag\\
    &+ 2 n y + 2 a n y + 6 b n y + 2 c n y + 12 n^2 y + c x y + 
    2 n x y)\big).\notag
        \end{align}
        \item For the parameter $c$ we obtain
        \begin{align}
            &(1 + b + c + 4 n) (2 + b + c + 4 n) (-1 + a + c + 2 n - y)\cdot F_{a,b,c,x,y}(a+2n;n)\\
            &-(2 + b + c + 4 n)  (a + c + 2 a c + b c + 2 c^2 + 2 n + 4 a n + 
              12 c n + 12 n^2 + a x + c x  \hphantom{a}\notag\\
            &+4 n x - y - a y - b y - 2 c y -  8  n  y)\cdot F_{a,b,c+1,x,y}(a+2n;n)\hphantom{a}\notag\\
            &+(1 + c + 2 n) (1 + a + b + c + 6 n) (1 + c + 2 n + x - y)\cdot F_{a,b,c+2,x,y}(a+2n;n)=0\notag
        \end{align}
        with certificate
        \[R_{F,c}(a,b,c,x,y,n,s)=\frac{(-1 + s) (1 + b + c + 4 n) (2 + b + c + 4 n) (1 - b - s + x)}{b + c + s + 4 n}.\]
        \item For the parameter $x$ we compute
        \begin{align}
            &(-1 + a + b + 2 n - x) (1 + 2 n + x)\cdot F_{a,b,c,x,y}(a+2n;n)\hphantom{a}\\
            &+(2 - a - 2 b + c - b c - 2 n - 4 b n - 4 n^2 + 4 x - a x - 2 b x +  c x + 2 x^2 - 3 y  \hphantom{a}\notag\\
            &+a y + b y - 2 x y)\cdot F_{a,b,c,x+1,y}(a+2n;n)\hphantom{a}\notag\\
            &+(1 + c + 2 n + x - y) (-1 + b + 2 n - x + y)\cdot F_{a,b,c,x+2,y}(a+2n;n)=0\notag
        \end{align}
        with certificate
        \[R_{F,x}(a,b,c,x,y,n,s)=\frac{(-1 + s) (-1 + b + c + s + 4 n) (1 - b - s + x) (2 n + y)}{(-2 + b +s + 2 n - x + y) (-1 + b + s + 2 n - x + y)}.\]
        \item Finally for the parameter $y$ we find the recursion
        \begin{align}
            &(-1 + c + 2 n + x - y) (1 + b + 2 n - x + y)\cdot F_{a,b,c,x,y}(a+2n;n)\hphantom{a}\\
            &+(6 - a + 2 b - 3 c - b c + 2 n - 4 c n - 4 n^2 - 4 x + a x + c x + 7 y - a y  \hphantom{a}\notag\\  
            &+b y - 2 c y - 2 x y + 2 y^2)\cdot F_{a,b,c,x,y+1}(a+2n;n)\hphantom{a}\notag\\
            &+(-3 + a + c + 2 n - y) (2 + 2 n + y)\cdot F_{a,b,c,x,y+2}(a+2n;n)=0\notag
        \end{align}
        with certificate
        \[R_{F,y}(a,b,c,x,y,n,s)=\frac{(-1 + s) (-1 + b + c + s + 4 n) (1 - b - s + x)}{1 + 2 n + y}.\]
    \end{itemize}
    Now, applying Lemma~\ref{recpreserv} to $I_{a,b,c,x,y}(a+2n;n)$ and Recursion~\eqref{arec} for the parameter $a$, we have that $I_{a,b,c,x,y}(a+2n;n)$ is of a shape as predicted in the proposition for all $a\geq 0$ if and only if $I_{0,b,c,x,y}(0+2n;n)$ and $I_{1,b,c,x,y}(1+2n;n)$ do so. 
    However, we know $I_{0,b,c,x,y}(0+2n;n)$ is of such a shape for arbitrary values of $b\geq 0$ if and only if it is true for $I_{0,0,c,x,y}(0+2n;n)$ and $I_{0,1,c,x,y}(0+2n;n)$ by applying again Lemma~\ref{recpreserv} to $I_{0,b,c,x,y}(0+2n;n)$ and Recursion~\eqref{brec} for parameter $b$.
    Continuing this idea we arrive at the observation that $I_{a,b,c,x,y}(a+2n;n)$ is of the shape
    \[(-1)^n\frac{(3n+1)!}{(n!)^3}\cdot \text{ rational function}(n)\]
    if it is true for all $I_{z_1,z_2,z_3,z_4,z_5}(z_1+2n;n)$ where $z=(z_1,z_2,z_3,z_4,z_5)$ varies over all possible choices of binary tuples $z\in\{0,1\}^5$.

    Naturally, the same holds for $F_{a,b,c,x,y}(a+2n;n)$. That is, the second statement\break of the proposition is true once it is shown for all $F_{z_1,z_2,z_3,z_4,z_5}(z_1+2n;n)$ where\break $z=(z_1,z_2,z_3,z_4,z_5)$ varies over all possible choices of binary tuples $z\in\{0,1\}^5$.

    Hence, the proof of the proposition breaks down to checking $2\cdot 2^5=64$ identities. This can again be achieved automatically by a computer via Zeilberger's algorithm and a good guessing method. Below, we present the layout of a code that does this for us. Our explicit implementation can be found in the Mathematica notebook accompanying this article. It uses again the function \texttt{Zb[...]} from \texttt{fastZeil.m} but also \texttt{Rate[...]} from Krattenthaler's guessing-algorithm package \texttt{rate2.m} \cite{AdvancedDet}. It can be obtained from Krattenthaler's website. This function takes a sequence of numbers and applies rational interpolation methods to guess a rational expression that interpolates this sequence. 

    The code works as follows.
    \begin{itemize}
        \item[(1)] We define functions $\texttt{Ifunk}:=I_{a,b,c,x,y}(a+2n;n)$ and \texttt{ISummand} where the second one captures only the symbolic expression for the summands in the definition of $I_{a,b,c,x,y}(a+2n;n)$. This is necessary, since the Zeilberger algorithm \texttt{Zb[...]} only takes the expression of the summand as its argument.
        \item[(2)] A five-times \texttt{For}-loop generates all possible choices for $z\in\{0,1\}^5$.
        \item[(3)] Then a variable \texttt{Ratio1} is defined as the rational function that interpolates the first 10 values of $I_{a,b,c,x,y}(a+2n;n)\cdot (-1)^n\frac{(n!)^3}{(3n+1)!}$ produced by \texttt{Rate[...]}. Sometimes, the guessed function does not interpolate some of the first initial values which does not matter since we prove our formulas only for $n$ large enough.
        \item[(4)] The variable \texttt{Ratio} then captures the expression stated in the proposition as a function of $n$. It remains to show $\texttt{Ifunk}(n)=\texttt{Ratio}(n)$.
        \item[(5)] To do so, we compare the first few values of both functions. Then apply Zeilberger's algorithm to $\texttt{Ifunk}(n)$ to obtain a linear recurrence with respect to $n$. In a final variable called \texttt{Rekursion} we define the simplified recurrence expression after replacing all instances of $\texttt{Ifunk}(n)$ with $\texttt{Ratio}(n)$ and print the result. If the output is \texttt{0} this means that $\texttt{Ratio}(n)$ fulfils the same recursion. Since both function coincide on enough initial values, we have equality for all $n$.
    \end{itemize}

Naturally, one has to repeat a similar procedure for $F_{a,b,c,x,y}(a+2n;n)$. One can find the explicit code executing these procedures in the Mathematica notebook accompanying this article.
This concludes the proof of the proposition. 
\end{proof}

Using the proposition, we are finally able to prove the main result of this subsection.
\begin{thm}[\sc{Reduction of the parameter $a$}]\label{pararedA}
    We have that $\mathbb{P}(a,b,c,x,y;n)$ with non-negative integer arguments $a,b,c,x,y\geq 0$ satisfies Theorem~\ref{1/3} if $\mathbb{P}(0,b,c,x,y;n)$ fulfils it. 
\end{thm}

\begin{proof}
    By Lemma~\ref{redA} we have
     \begin{align}
        \mathbb{P}(a,b,c,x,y;n)&-\mathbb{P}(a-1,b,c,x,y;n) \label{diffeq}\\
        &=\frac{(c+2n)!}{(b+2n+1)_{c+2n}}I_{a,b,c,x,y}(a+2n;n)\cdot F_{a,b,c,x,y}(a+2n;n).\notag
    \end{align}
    Now, having a closer look at the factors on the right-hand side we see
    \begin{align}
        \frac{(c+2n)!}{(b+1+2n)_{c+2n}}=\frac{(c+2n)!}{\frac{(b+1+c+4n)!}{(b+2n)!}}=\frac{(2n)!}{\frac{(4n)!}{(2n)!}}q_1(n)=\frac{(2n)!(2n)!}{(4n)!}q_1(n)\notag
    \end{align}
    for some rational function $q_1(p)$ since for a positive integer constants $C$ and $k$ we always have $(C+k\cdot n)!=(k\cdot n)!\cdot \text{polynomial}(n)$. Moreover, Proposition~\ref{IFshape} yields that there exist rational functions $q_2(n)$ and $q_3(n)$ such that
    \begin{align}
        &\frac{(c+2n)!}{(b+2n+1)_{c+2n}}I_{a,b,c,x,y}(a+2n;n)\cdot F_{a,b,c,x,y}(a+2n;n)\notag\\
        &=\frac{(2n)!(2n)!}{(4n)!}q_1(n)\cdot (-1)^n\frac{(3n+1)!}{(n!)^3}\cdot q_2(n)\cdot (-1)^n\frac{(2n)!(3n+1)!(4n)!}{(n!)^3(6n+2)!}\cdot q_3(n)\notag\\
        &=\frac{(2n)!(2n)!(2n)!}{(n!)^6}\cdot \frac{(3n+1)!(3n+1)!}{(6n+2)!}\big(q_1(n)\cdot q_2(n)\cdot q_3(n)\big)\notag\\
        &=q(n)\cdot \binom{2n}{n}^3\bigg/\binom{6n+2}{3n+1}\notag
    \end{align}
    where $q(n)=q_1(n)\cdot q_2(n)\cdot q_3(n)$ is again a rational function in $n.$ Hence, if we insert this into Equation~\eqref{diffeq} we instantly obtain
    \[\mathbb{P}(a,b,c,x,y;n)=\mathbb{P}(a-1,b,c,x,y;n)+q(n)\cdot \binom{2n}{n}^3\bigg/\binom{6n+2}{3n+1}\]
    which implies by induction on the parameter $a$ that $\mathbb{P}(a,b,c,x,y;n)$ fulfils Theorem~\ref{1/3} if $\mathbb{P}(0,b,c,x,y;n)$ does.
\end{proof}

\subsection{Reduction of the parameter $b$}
As we will see, the hardest part is already over. Once we know how to reduce the first variable, the others follow by straightforward reflection and rotation arguments.
\begin{thm}[\sc{Reduction of the parameter $b$}]\label{pararedB}
    If\/ $\mathbb{P}(a,0,c,x',y';n)$ fulfils Theorem~\ref{1/3} for all positions $x',y'\in \Z$, then also $\mathbb{P}(a,b,c,x,y;n)$ satisfies the theorem for all positions $x,y\in\Z$.
\end{thm}

\begin{proof}
    As it is shown in Figure~\ref{breflect} the tilings of the semiregular hexagon $H_{a,b,c}$ with a fixed lozenge at position $(x,y)$ are in bijection with the tilings of the semiregular hexagon $H_{b,a,c}$ with a fixed lozenge at some other location $(x',y')$ by reflecting along a vertical line.

    Therefore, we have
    \[\mathbb{P}(a,b,c,x,y;n)=\mathbb{P}(b,a,c,x',y';n)\]
    for some $x',y'\in\Z$. However, by Theorem~\ref{pararedA} we know $\mathbb{P}(b,a,c,x',y';n)$ fulfils Corollary~\ref{1/3} if $\mathbb{P}(0,a,c,x',y';n)$ does. Moreover, since again
    \[\mathbb{P}(0,a,c,x',y')=\mathbb{P}(a,0,c,x'',y'';n)\]
    for some location $(x'',y'')\in\Z^2$ the assertion follows.
    \begin{figure}
\centering
\includegraphics[width=0.8\textwidth]{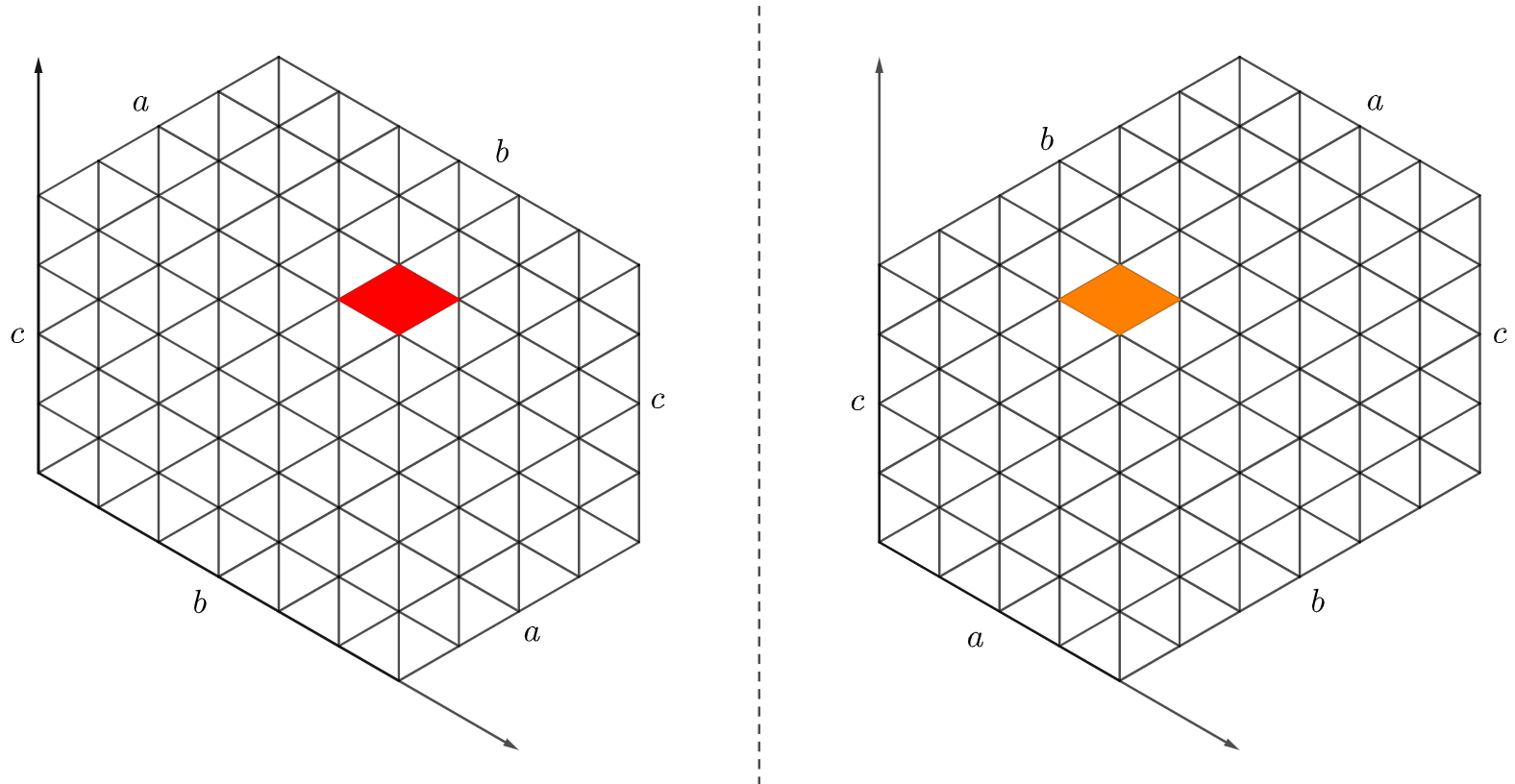}
\caption{Tilings of $H_{a,b,c}$ with a fixed horizontal lozenge are in bijection to tilings of $H_{b,a,c}$ with a fixed horizontal lozenge.}
\label{breflect}
\end{figure}
\end{proof}

Thus, we are able to already reduce two out of the five parameters.

\begin{cor}
    If\/ $\mathbb{P}(0,0,c,x',y')$ satisfies Theorem~\ref{1/3} for all positions $(x',y')\in\Z$ then so does $\mathbb{P}(a,b,c,x,y;n)$ for all $a,b,x,y\in\Z$.
\end{cor}

\begin{proof}
    Combine Theorem~\ref{pararedA} with Theorem~\ref{pararedB}.
\end{proof}

\subsection{Reduction of the parameter $c$}
Finally, we will see how we can reduce the problem to the regular hexagon by sending all parameters to $0$ simultaneously.
\begin{thm}[\sc{Reduction of the parameter $c$}]\label{pararedC}
    If\/ $\mathbb{P}(0,0,0,x',y';n)$ satisfies Theorem~\ref{1/3} for all locations $(x',y')\in \Z$ then so does $\mathbb{P}(a,b,c,x,y;n)$ with $a,b,c,x,y\in\Z$.
\end{thm}

\begin{proof}
    First of all, we already have seen that $\mathbb{P}(a,b,c,x,y;n)$ satisfies the theorem if\break $\mathbb{P}(0,0,c,x',y';n)$ does for all positions $(x',y')\in \Z^2$. Therefore let us assume, that we already have $a=0$ and $b=0$.
    
    Inside the triangular grid, observe that in a lozenge tiling of $H_{a,b,c}$ every triangle $T$ inside the semiregular hexagon has to be covered either by a horizontal lozenge, a left-leaning lozenge or a right-leaning lozenge. In terms of probabilities when sampling tilings uniformly at random this means for every such triangle $T$ we have
    \[\mathbb{P}\big(\text{\lozHlow\ covers } T\big)+\mathbb{P}\big(\text{\ \lozL \ covers }T\big)+\mathbb{P}\big(\text{\ \lozR \ covers }T\big)=1.\]
    However, also the fixed left-leaning or right-leaning lozenge corresponds to a horizontal lozenge in a rotated hexagon. See Figure~\ref{rotC} for an visualisation of the idea.
    \begin{figure}
    \centering
    \includegraphics[width=0.4\textwidth]{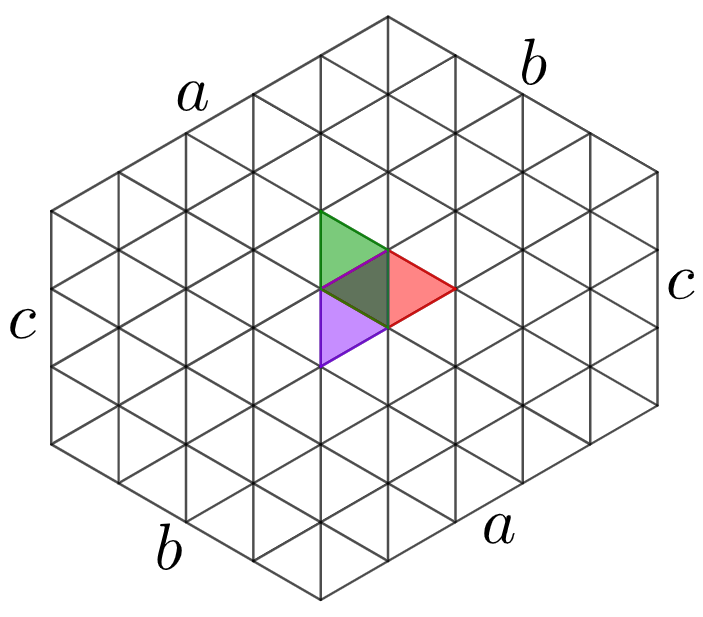}
    \caption{The red lozenge is horizontal in $H_{a,b,c}$. However, its green counterpart is horizontal in the rotated hexagon $H_{b,c,a}$ and the purple hexagon is horizontal in $H_{c,a,b}$}
    \label{rotC}
    \end{figure}
    In conclusion, if some horizontal lozenge is located at $(x,y)\in\Z^2$ inside the semiregular hexagon $H_{0,0,c}$ then there also exist locations $(\bar{x},\bar{y})$ and $(\tilde{x},\tilde{y})\in\Z^2$ such that
    \[\mathbb{P}(0,0,c,x,y;n)+\mathbb{P}(0,c,0,\bar{x},\bar{y};n)+\mathbb{P}(c,0,0,\tilde{x},\tilde{y};n)=1.\]
    From this equation it is easy to see that $\mathbb{P}(0,0,c,x,y;n)$ fulfils the theorem once this is shown for $\mathbb{P}(0,c,0,\bar{x},\bar{y};n)$ and $\mathbb{P}(c,0,0,\tilde{x},\tilde{y};n)$. However, by Theorem~\ref{pararedB} and Theorem~\ref{pararedA} both $\mathbb{P}(0,c,0,\bar{x},\bar{y};n)$ and $\mathbb{P}(c,0,0,\tilde{x},\tilde{y};n)$ satisfy Theorem~\ref{1/3} if the placement probability $\mathbb{P}(0,0,0,x',y';n)$ does it for all locations $(x',y')\in\Z^2$ as proposed. In particular, this also finishes the proof of the main result of this section.
\end{proof}

In conclusion, we have shown that once the $1/3$-phenomenon is proven in regular hexagons it also remains valid in arbitrary semiregular hexagons.

\section{Reduction to the centre and proof of the phenomenon} \label{centreproof}

In this section we finish our proof of Theorem~\ref{1/3}. We follow a similar line of ideas as in \cite{1/4} in the case of the Aztec diamond. The goal is to further reduce the problem of computing $\mathbb{P}(0,0,0,x,y;n)$ to the case of the origin where  parameter $x$ and $y$ are equal to zero. For this particular position, the 1/3-phenomenon has already been verified by Fulmek and Krattenthaler~\cite[Cor. 7 (1.8)]{CenterCase}.
This final reduction is achieved by using  Koutschan's Mathematica package \texttt{HolonomicFunctions.m} to compute linear recurrences of $\mathbb{P}(0,0,0,x,y;n)$ in the variables $x$ and $y$. In particular, we are going to use Creative Telescoping to do so. First, let
\begin{align}
    \mathbb{P}(x,y;n):={}& \mathbb{P}(0,0,0,x,y;n)\notag\\
    ={}&\frac{(2n)!(2n)!}{(4n)!}\sum_{j=1}^{2n}\sum_{i=1}^{j}\sum_{s=1}^j (-1)^{i+s}\binom{j-1}{s-1}\binom{2n+i+x-y-2}{x+2n-1}\notag\\
    &\cdot \binom{2n+s-x+y-1}{s-x-1}\frac{(2n+1)_{s-1}(2n+1)_{i-1}(4n+i)_{j-i}}{(4n+1)_{s-1}(i-1)!(j-i)!}.\notag
\end{align}
We want to apply Creative Telescoping to the summand given by
\begin{align}
    H(x,y,n,j,i,s)={}&(-1)^{i+s}\binom{j-1}{s-1}\binom{2n+i+x-y-2}{x+2n-1}\notag\\
    &\cdot \binom{2n+s-x+y-1}{s-x-1}\frac{(2n+1)_{s-1}(2n+1)_{i-1}(4n+i)_{j-i}}{(4n+1)_{s-1}(i-1)!(j-i)!}.\notag
\end{align}

Now, we apply the Creative Telescoping algorithm via
\begin{align}
    \texttt{Crea}&\texttt{tiveTelescoping[}\notag\\
    &\texttt{Annihilator[}H(x,y,n,j,i,s),\{S_x,S_j,S_i,S_s\}\texttt{]},S_s-1,\{S_x,S_j,S_i\}\texttt{]}\notag
\end{align}
As already explained in Section \ref{ZAHA}, \texttt{Annihilator[...]} produces a Gröbner basis for the left-ideal of operators $Ann(x,y,n,j,i,s,S_x,S_j,S_i,S_s)$ in $\Q(x,y,n,j,i,s)[S_x,S_j,S_i,S_s]$ (the polynomials of shift operators $S_x,S_j,S_i$ or $S_s$ over the field of rational functions in $x,y,n,j,i,s$) such that
\[Ann(x,y,n,j,i,s,S_x,S_j,S_i,S_s)H(x,y,n,j,i,s)=0.\]
Furthermore, \texttt{CreativeTelescoping[...]} yields as output two lists $\{A_1,A_2,A_3\}$ and $\{B_1,B_2,B_3\}$ of operators of different dependencies such that
\begin{align}
    A_k(x,y,n,j,i,S_x,S_j,S_i)&H(x,y,n,j,i,s)\notag\\
    &{}+(S_s-1)B_k(x,y,n,j,i,s,S_x,S_j,S_i,S_s)H(x,y,n,j,i,s)=0 \label{Tele1}
\end{align}
for $k=1,2,3$. 
\begin{rem}
    Please notice that all operators $A_i$, $B_i$ but also $C_i,D_i$ and $E$ below are explicit polynomials in the shift operators. However, their sheer size makes it hard to write them down properly. Their existence and structure are secured by the algorithms and if one wants to further examine them, they can easily be reproduced. 
\end{rem}

If we now have a closer look at $H$ we observe that
\[\sum_{s=1}^jH(x,y,n,j,i,s)=\sum_{s=0}^\infty H(x,y,n,j,i,s)\]
and also
\[\sum_{i=1}^jH(x,y,n,j,i,s)=\sum_{i=0}^\infty H(x,y,n,j,i,s).\]
Then, if we take the sum of the Expressions~\eqref{Tele1} over all values for $s$ we obtain
\begin{align}
    A_k(x,y,n,j,i&,S_x,S_j,S_i)\sum_{s=0}^\infty H(x,y,n,j,i,s)\notag\\
    &{}+\sum_{s=0}^\infty(S_s-1)B_k(x,y,n,j,i,s,S_x,S_j,S_i,S_s)H(x,y,n,j,i,s)=0 \notag
\end{align}
for $k=1,2,3$. A small analysis of the operators $B_k$, and the observation that \break $H(x,y,n,j,i,s)$ equals $0$ for $s=0$ or $s$ large enough, tells us that the second sum is telescoping and vanishes. Hence, we have
\[A_k(x,y,n,j,i,S_x,S_j,S_i)\sum_{s=0}^\infty H(x,y,n,j,i,s)=0\]
for $k=1,2,3$. In a next step, we repeat this procedure. Namely we look for a Creative Telescoping expression in the ideal spanned by $A_1, A_2$ and $A_3$. Thus, we apply
\[\texttt{CreativeTelescoping[}\{A_1,A_2,A_3\},S_i-1,\{S_x,S_j\}\texttt{]}.\]
This then produces operators $C_1,C_2,C_3$ and $D_1,D_2,D_3$ such that
\begin{align}
    C_k(x,y,n,j,S_x,S_j)\sum_{s=0}^\infty &H(x,y,n,j,i,s)\notag\\
    &+(S_i-1)D_k(x,y,n,j,i,S_x,S_j,S_i)\sum_{s=0}^\infty H(x,y,n,j,i,s)=0 \label{Tele2}
\end{align}
for $k=1,2,3$. Now, notice that $D\cdot \sum H$ is of the following shape:
\begin{align}
    D_k(x,y,n,&j,i,S_x,S_j,S_i)\sum_{s=0}^\infty H(x,y,n,j,i,s)\notag\\
    &=\sum_{l=1}^Lq_l(x,y,n,j,i)\sum_{s=0}^\infty H(x+x_l,y,n,j+j_l,i+i_l,s)\notag
\end{align}
for some rational functions $q_l(x,y,n,j,i)$ and some non-negative integers $L, x_1,\dots, x_L$,\break $j_1,\dots,j_L$ and $i_1,\dots,i_L$. However, a small analysis of the $D_k$ shows that they all have polynomial degree in $S_i$ equal to $0$. Therefore, $i_1=i_2=\dots=i_L=0$. Hence, if we take the sum of Expressions~\eqref{Tele2} over all values of $i\in\N\cup\{0\}$ the second summand telescopes and vanishes again. In conclusion, by our observation from before we obtain 
\begin{align}
    \sum_{i=0}^\infty C_k(x,y,n,j,S_x,S_j)&\sum_{s=0}^\infty H(x,y,n,j,i,s)\notag\\
    &=C_k(x,y,n,j,S_x,S_j)\sum_{i=0}^\infty\sum_{s=0}^\infty H(x,y,n,j,i,s)\notag\\
    &=C_k(x,y,n,j,S_x,S_j)\sum_{i=1}^j\sum_{s=1}^j H(x,y,n,j,i,s)=0\notag
\end{align}
for $k=1,2,3$. Finally, we apply Creative Telescoping a third time. We are looking for a telescoping expression with respect to the variable $j$. If we run
\[\texttt{CreativeTelescoping[}\{C_1,C_2,C_3\},S_j-1,\{S_x\}\texttt{]}\]
then after several hours we obtain the two operators $(S_x-1)$ and $E(x,y,n,j,S_x,S_j)$ such that
\begin{align}
    (S_x-1)\sum_{i=1}^j\sum_{s=1}^j &H(x,y,n,j,i,s)\notag\\
    &+(S_j-1)E(x,y,n,j,S_x,S_j)\sum_{i=1}^j\sum_{s=1}^j H(x,y,n,j,i,s)=0. \label{Tele3}
\end{align}
The function $E(x,y,n,j,S_x,S_j)$ is rational in $x,y,n,j$ and polynomial in $S_x$ and $S_j$. In particular, it is linear in $S_x$ and quadratic in $S_j$. Hence, there exist rational functions $q_{jj}(x,y,n,j)$, $q_{j}(x,y,n,j)$, $q_x(x,y,n,j)$ and $q_1(x,y,n,j)$ such that
\begin{align}
    E(x,y,n,j,S_x,&S_j)\notag\\
    &=q_{jj}(x,y,n,j)\cdot S_j^2+q_{j}(x,y,n,j)\cdot S_j+q_{x}(x,y,n,j)\cdot S_x+q_{1}(x,y,n,j).\notag
\end{align}
In conclusion we have
\begin{align}
     E(x,&y,n,j,S_x,S_j)\sum_{i=1}^j\sum_{s=1}^j H(x,y,n,j,i,s)\notag\\
     &=q_{jj}(x,y,n,j)\cdot \sum_{i=1}^{j+2}\sum_{s=1}^{j+2} H(x,y,n,j+2,i,s)\notag\\
     &\hphantom{aaaa}+q_{j}(x,y,n,j)\cdot  \sum_{i=1}^{j+1}\sum_{s=1}^{j+1} H(x,y,n,j+1,i,s)\notag\\
     &\hphantom{aaaaaaaa}+q_{x}(x,y,n,j)\cdot \sum_{i=1}^j\sum_{s=1}^j H(x+1,y,n,j,i,s)\notag\\
     &\hphantom{aaaaaaaaaaaa}+q_1(x,y,n,j)\sum_{i=1}^j\sum_{s=1}^j H(x,y,n,j,i,s)=:G(x,y,n,j)\notag.
\end{align}

If we now take the sum of the Expressions~\eqref{Tele3} over $j=-3,-2,-1,0,1,2,\dots,2n$ and multiply everything by $\frac{(2n)!(2n)!}{(4n)!}$ we obtain by telescoping
\begin{align}
    \mathbb{P}(x+1,y;n)-\mathbb{P}(x,y;n)+\frac{(2n)!(2n)!}{(4n)!}\Big(G(x,y,n,2n+1)-G(x,y,n,-3)\Big)=0. \label{xrec}
\end{align}
Notice that the additional summands for non-positive instances of $j$ all vanish. Therefore we also have $G(x,y,n,-3)=0$.
Let us now have a closer look at $G(x,y,n,2n+1)$. We have
\begin{align}
    G(x,y,n,2n+1)&=q_{jj}(x,y,n,2n+1)\cdot \sum_{i=1}^{2n+3}\sum_{s=1}^{2n+3} H(x,y,n,2n+3,i,s)\notag\\
     &\hphantom{aaaa}+q_{j}(x,y,n,2n+1)\cdot  \sum_{i=1}^{2n+2}\sum_{s=1}^{2n+2} H(x,y,n,2n+2,i,s)\notag\\
     &\hphantom{aaaaaaaa}+q_{x}(x,y,n,2n+1)\cdot \sum_{i=1}^{2n+1}\sum_{s=1}^{2n+1} H(x+1,y,n,{2n+1},i,s)\notag\\
     &\hphantom{aaaaaaaaaaaa}+q_1(x,y,n,2n+1)\sum_{i=1}^{2n+1}\sum_{s=1}^{2n+1} H(x,y,n,{2n+1},i,s).\notag
\end{align}

However, for $a\in\{1,2,3\}$ we have already seen in the previous section
\begin{align}
    \sum_{i=1}^{2n+a}\sum_{s=1}^{2n+a} H(x,y,n,{2n+a},i,s)=I_{a,0,0,x,y}(a+2n;n)\cdot F_{a,0,0,x,y}(a+2n;n).\notag
\end{align}
For the right-hand side, we have already shown in Proposition \ref{IFshape} the existence of rational functions $q_{a,0,0,x,y}(n)$ and $r_{a,0,0,x,y}(n)$ such that
\begin{align}
    I_{a,0,0,x,y}(a+2n;n)&=(-1)^n\frac{(3n+1)!}{(n!)^3}\cdot q_{a,0,0,x,y}(n) \ \ \text{ and }\notag \\
    F_{a,0,0,x,y}(a+2n;n)&=(-1)^n\frac{(2n)!(3n+1)!(4n)!}{(n!)^3(6n+2)!}\cdot r_{a,0,0,x,y}(n).\notag
\end{align}
Therefore, we already have 
\[\frac{(2n)!(2n)!}{(4n)!}G(x,y,n,2n+1)=g_{x,y}(n)\cdot \binom{2n}{n}^3\bigg/\binom{6n+2}{3n+1}\]
for some rational function $g_{x,y}$, for all choices of values for $x$ and $y$, and $n$ large enough. In combination with the Recurrence~\eqref{xrec} and Lemma~\ref{recpreserv} this suffices to deduce the following corollary.
\begin{cor}
    Let $x,y\in \Z$ be arbitrary. Then $\mathbb{P}(x,y;n)$ satisfies the $1/3$-phenomenon if $\mathbb{P}(0,y;n)$ does. 
\end{cor}

Finally, we play the same game again with respect to the parameter $y$. However, we are now allowed to set $x=0$ to improve the efficiency. First we run

\begin{align}
    \texttt{Crea}&\texttt{tiveTelescoping[}\notag\\
    &\texttt{Annihilator[}H(0,y,n,j,i,s),\{S_y,S_j,S_i,S_s\}\texttt{]},S_s-1,\{S_y,S_j,S_i\}\texttt{]}\notag
\end{align}

This returns operators $A_1',A'_2,A_3'$ and $B_1',B_2',B_3'$ such that
\begin{align}
      A_k'(y,n,j,i,S_y,S_j,S_i)&H(0,y,n,j,i,s)\notag\\
    &{}+(S_s-1)B_k'(y,n,j,i,s,S_y,S_j,S_i,S_s)H(0,y,n,j,i,s)=0 \notag
\end{align}
for $k=1,2,3$. For the same reasons as before, if we now take the sum over all non-negative values of $s$, the second part telescopes and vanishes and what remains is
\[A_k'(y,n,j,i,S_y,S_j,S_i)\sum_{s=0}^\infty H(0,y,n,j,i,s)=0\]
for $k=1,2,3$. Moving onward, we apply Creative Telescoping again:
\[\texttt{CreativeTelescoping[}\{A_1',A_2',A_3'\},S_i-1,\{S_y,S_j\}\texttt{]}.\]
This yields operators $C'_1,C'_2,C_3'$ and $D'_1,D'_2,D_3'$ such that 
\begin{align}
    C_k'(y,n,j,S_y,S_j)\sum_{s=0}^\infty &H(0,y,n,j,i,s)\notag\\
    &+(S_i-1)D_k'(y,n,j,i,S_y,S_j,S_i)\sum_{s=0}^\infty H(0,y,n,j,i,s)=0\notag
\end{align}
for $k=1,2,3$. Analogously to before, the second part vanishes when computing the sum over all non-negative values for $i$. I.e. we arrive at
\[C_k'(y,n,j,S_y,S_j)\sum_{i=1}^j\sum_{s=1}^jH(0,y,n,j,i,s)=0\]
for $k=1,2,3$. In the end, we apply Creative Telescoping one final time:
\[\texttt{CreativeTelescoping[}\{C_1',C_2',C_3'\},S_j-1,\{S_y\}\texttt{]}.\]
We obtain the operators $(S_y-1)$ and $E'(y,n,j,S_y,S_j)$ such that
\begin{align}
    (S_y-1)\sum_{i=1}^j\sum_{s=1}^j &H(0,y,n,j,i,s)\notag\\
    &+(S_j-1)E'(y,n,j,S_y,S_j)\sum_{i=1}^j\sum_{s=1}^j H(0,y,n,j,i,s)=0. \label{TeleY3}
\end{align}
Now, $E'(y,n,j,S_y,S_j)$ is of similar shape as $E(x,y,n,j,S_x,S_j)$ in the sense that there exist rational functions $r_{jj}(y,n,j),r_j(y,n,j),r_y(y,n,j)$ and $r_1(y,n,j)$ such that
\begin{align}
     E'(&y,n,j,S_y,S_j)\sum_{i=1}^j\sum_{s=1}^j H(0,y,n,j,i,s)\notag\\
     &=r_{jj}(y,n,j)\cdot \sum_{i=1}^{j+2}\sum_{s=1}^{j+2} H(0,y,n,j+2,i,s)\notag\\
     &\hphantom{aaaa}+r_{j}(x,y,n,j)\cdot  \sum_{i=1}^{j+1}\sum_{s=1}^{j+1} H(x,y,n,j+1,i,s)\notag\\
     &\hphantom{aaaaaaaa}+r_{y}(x,y,n,j)\cdot \sum_{i=1}^j\sum_{s=1}^j H(0,y+1,n,j,i,s)\notag\\
     &\hphantom{aaaaaaaaaaaa}+r_1(y,n,j)\sum_{i=1}^j\sum_{s=1}^j H(0,y,n,j,i,s)=:G'(y,n,j)\notag.
\end{align}
Again, we sum over $j=-3,-2,-1,0,1,\dots, 2n$ in Expression~\eqref{TeleY3} and multiply by $\frac{(2n)!(2n)!}{(4n)!}$ to arrive at
\begin{align}
    \mathbb{P}(0,y+1;n)-\mathbb{P}(0,y;n)+\frac{(2n)!(2n)!}{(4n)!}\cdot G'(y,n,2n+1). \notag
\end{align}

Analogously to before, we deduce
\[\frac{(2n)!(2n)!}{(4n)!}\cdot G'(y,n,2n+1)=g'_{y}(n)\cdot \binom{2n}{n}^3\bigg/\binom{6n+2}{3n+1}\]
for some rational function $g'_y(n)$ in $n$. If we sum up all our observations until now, the following statement is seen to be true.
\begin{cor}
    Let $x,y\in \Z$ be arbitrary. Then $\mathbb{P}(x,y;n)$ satisfies the $1/3$-phenomenon if the phenomenon holds for $\mathbb{P}(0,0;n)$.
\end{cor}

However, this finishes the proof for the $1/3$-phenomenon, Theorem~\ref{1/3}, since many individual instances for $\mathbb{P}(a,b,c,x,y,;n)$ have already been analysed. In particular, in \cite[Cor. 7 (1.8)]{CenterCase} Fulmek and Krattenthaler already computed that
\[\mathbb{P}(0,0;n)=\mathbb{P}(0,0,0,0,0;n)=\frac{1}{3}-\frac{6n+1}{6(3n+1)}\cdot \binom{2n}{n}^3\Big/\binom{6n+2}{3n+1}\]
which clearly aligns with the shape claimed by Theorem~\ref{1/3}. Thus, the phenomenon holds for all $a,b,c,x,y\in \Z$ and all $n\in\N$ large enough. \qed

\newpage

\bibliographystyle{plain}
\bibliography{bib}

\end{document}